\documentclass[12pt]{article}

\usepackage{amsmath, amssymb, latexsym, amsthm}
\usepackage{fullpage}
\usepackage{color}
\usepackage{tikz}
\usepackage{graphicx}
\usepackage{stmaryrd,amsbsy,enumerate}
\usepackage{hyperref}
\usepackage{mathscinet}
\usetikzlibrary{calc}
\usepackage{cancel}
\usepackage{xspace}
\usepackage[mathlines]{lineno}
\usepackage{url}

\usepackage[obeyFinal,colorinlistoftodos]{todonotes}

\if10     
\usepackage[mathlines]{lineno}
\newcommand*\patchAmsMathEnvironmentForLineno[1]{%
  \expandafter\let\csname old#1\expandafter\endcsname\csname #1\endcsname
  \expandafter\let\csname oldend#1\expandafter\endcsname\csname end#1\endcsname
  \renewenvironment{#1}%
     {\linenomath\csname old#1\endcsname}%
     {\csname oldend#1\endcsname\endlinenomath}}%
\newcommand*\patchBothAmsMathEnvironmentsForLineno[1]{%
  \patchAmsMathEnvironmentForLineno{#1}%
  \patchAmsMathEnvironmentForLineno{#1*}}%
\AtBeginDocument{%
\patchBothAmsMathEnvironmentsForLineno{equation}%
\patchBothAmsMathEnvironmentsForLineno{align}%
\patchBothAmsMathEnvironmentsForLineno{flalign}%
\patchBothAmsMathEnvironmentsForLineno{alignat}%
\patchBothAmsMathEnvironmentsForLineno{gather}%
\patchBothAmsMathEnvironmentsForLineno{multline}%
}
\linenumbers
\fi

\newtheorem{theorem}{Theorem} 
\newtheorem{theorem*}{Theorem} 

\newtheorem{lemma}[theorem]{Lemma}

\newtheorem{claim}[theorem]{Claim}

\renewcommand{\epsilon}{\varepsilon}

\theoremstyle{definition}

\theoremstyle{remark}

\newcommand{\URL}{\url{http://orion.math.iastate.edu/lidicky/pub/c5k3free/}}

\renewcommand{\P}{\mathbb{P}}

\newcounter{arxiv}

\begin{document}
\setcounter{arxiv}{0}


\title{Pentagons in triangle-free graphs}
\author{
Bernard Lidick\'{y}\thanks{Department of Mathematics, Iowa State University, Ames, IA, E-mail: {\tt lidicky@iastate.edu}. Research of this author is supported in part by NSF grant DMS-1600390.}
\and
Florian Pfender\thanks{Department of Mathematical and Statistical Sciences, University of Colorado Denver, E-mail: {\tt 
Florian.Pfender@ucdenver.edu}. Research is partially supported in part by NSF grant DMS-1600483.} 
}

\maketitle

\begin{abstract}
For all $n\ge 9$, we show that the only triangle-free graphs  on $n$ vertices
maximizing the number $5$-cycles are balanced blow-ups of a 5-cycle.
This completely resolves a conjecture by Erd\H{o}s, and extends results by Grzesik and Hatami, Hladk\'y, Kr\'{a}l', Norin, and Razborov, where they independently showed this same result for large $n$ and for all $n$ divisible by $5$.
\end{abstract}

\section{Introduction}

In 1984, Erd\H{o}s~\cite{Erdos1984} conjectured that for all $n\ge 5$, the balanced blow-up of a $5$-cycle maximizes the number of $5$-cycles in the class of triangle-free graphs on $n$ vertices. {\em Balanced blow-up} means in this context that every vertex of the $5$-cycle is replaced by an independent set of size $\lfloor\frac{n}5\rfloor$ or  $\lceil\frac{n}5\rceil$, and edges are replaced by complete bipartite graphs between the sets.
This conjecture, in a way, supports the Meta-Theorem saying that among all triangle-free graphs, this blow-up of $C_5$ is the "least" bipartite graph.

In 2012, Grzesik~\cite{Grzesik}, and independently in 2013, Hatami, Hladk\'y, Kr\'{a}l', Norin, and Razborov~\cite{HatamiHKNR13} settled this conjecture asymptotically by showing that any $n$-vertex triangle-free graph has at most $\frac{5!}{5^5}{n\choose 5}(1+o(1))=0.0384{n\choose 5}(1+o(1))$ $5$-cycles. Furthermore, the second group of authors showed uniqueness of the extremal graphs for most values of $n$.
\begin{theorem}[Hatami et. al. \cite{HatamiHKNR13}]\label{prev}
Let $n$ be either divisible by $5$, or large enough. The maximum number of copies of a $5$-cycle in triangle-free graphs
on $n$ vertices is
\[
\prod_{i=0}^4\left\lfloor\frac{n+i}{5}\right\rfloor.
\]
Moreover, the only triangle-free graphs  on $n$ vertices
maximizing the number $5$-cycles are balanced blow-ups of a 5-cycle.
\end{theorem}
Both papers
\cite{Grzesik} and \cite{HatamiHKNR13}
 use the theory of flag algebras, a theory developed by Razbarov~\cite{Raz07}, which can find inequalities of subgraph densities in graph limits with the help of semi-definite programming.

In 2011, Michael~\cite{Michael11} observed that the M\"obius ladder on $8$ vertices (i.e. the $8$-cycle with all diagonals added) contains the same number of $5$-cycles as the balanced blow-up of $C_5$.
In this note we resolve the remaining cases of the conjecture. We also employ flag algebra techniques to find helpful inequalities, and then we use stability results. Finally, enumeration is used to deal with $n \leq 9$. Note that our analysis to show uniqueness is significantly simpler than the analysis in~\cite{HatamiHKNR13}.
\begin{theorem}\label{main}
For all $n$, the maximum number of copies of a $5$-cycle in any triangle-free graph
on $n$ vertices is
\[
\prod_{i=0}^4\left\lfloor\frac{n+i}{5}\right\rfloor.
\]
Moreover, the only triangle-free graphs  on $n\ge 5$ vertices
maximizing the number $5$-cycles are balanced blow-ups of a $5$-cycle, and the M\"obius ladder $ML_8$ for the special case of $n=8$.
\end{theorem}

For the ease of presentation, we will use a simplified language of graph limits. No understanding past the following definitions is required to follow this note. A {\em graphon $B$ of a graph $G$} with vertex set $V(G)=\{v_1,v_2,\ldots,v_n\}$ is a symmetric function  $b:[0,1]^2 \to \{0,1\}$ with
\begin{align*}
b(x,y)=&\begin{cases}
1,& \mbox{ if }\frac{i-1}{n}\le x<\frac{i}{n},\frac{j-1}{n}\le y<\frac{j}{n},v_iv_j\in E(G),\\
0,& \mbox{ if }\frac{i-1}{n}\le x<\frac{i}{n},\frac{j-1}{n}\le y<\frac{j}{n},v_iv_j\notin E(G).
\end{cases}
\end{align*}
For a set $X=\{x_1,x_2,\ldots,x_k\}$ of real numbers in $[0,1]$, $B(X)$ is the graph with vertex set $X$ and edge set
$\{\{x_i,x_j\}:B(x_i,x_j)=1\}$.
For a graph $H$, the {\em induced density} of $H$ in $B$ is
\[
d(H)=d_B(H)=\P(B(X)\cong H),
\]
where $X$ is chosen uniformly at random from $[0,1]^{|H|}$. Note that this quantity equals the limit of the densities of $H$ in balanced blow-ups of $G$ on $N$ vertices, where $N\to\infty$. This notion of a graphon agrees with the theory of dense graph limits developed by Lov\'asz and Szegedy (\cite{LS06}, see the book by Lov\'asz~\cite{Lovasz} for a thorough introduction to the theory). We use a different parametrization to follow the theory of flag algebras by Razborov.

In the following, we assume some familiarity with flag algebras,
and we provide only a very brief reminder of the definitions.
For a proper introduction to flag algebras, see~\cite{Raz07}.

We will need the notion of a flag. A {\em flag} is a graph $F$ on $n$ vertices, in which some vertices, say the first $k$ vertices $X=\{x_1,\ldots,x_k\}\subseteq V(F)$, are labeled with distinct labels. The {\em type} $\sigma$ of $F$ is the labeled graph induced on $X$. Two flags are isomorphic if there exists an isomorphism of the underlying graphs which induces an isomorphism of the labeled vertices. The graphon $B$ of the flag $F$ is the graphon of the unlabeled graph underlying $F$, together with the set $Y=\{\frac{0}{n},\frac{1}{n},\ldots,\frac{k-1}{n}\}\in [0,1]^{|X|}$. Note that $B(Y)=\sigma$. For any flag $H$ of type $\sigma$, we then have
\[
d(H)=d_B(H)=\P(B(X\cup Y)\cong H),
\]
where $X$ is chosen uniformly at random from $[0,1]^{|H|-|Y|}$. 

A {\em Flag Algebra} $\mathcal{F}^\sigma$ consists of formal linear combinations of flags of the type $\sigma$, with the canonical definitions for addition and scalar multiplication. A multiplication of flags is also defined in a way that it naturally corresponds to the multiplication of flag densities in graphons of flags.

\section{Proof of Theorem~\ref{main}}

The core of the proof of Theorem~\ref{main} is the following lemma.

\begin{lemma}\label{core}
If $G$ is a triangle-free graph maximizing the number of 5-cycles on $n \geq 10$ vertices,
then there exists a 5-cycle $C$ in $G$ such that every other vertex of $G$ has exactly two neighbors in $C$.
\end{lemma}

Before we prove the lemma, we first show how it implies Theorem~\ref{main}. 

\begin{proof}[Proof of Theorem~\ref{main}]
Let $G$ be a triangle-free graph maximizing the number of $C_5$s on $n$ vertices.
For $n \leq 9$, we use a computer to enumerate all $2480$ triangle-free graphs on 
$n$ vertices and check that the theorem is true.
Hence we assume $n \geq 10$ and we can use Lemma~\ref{core}.

Let $C=v_1v_2v_3v_4v_5v_1$ 
be the 5-cycle  from Lemma~\ref{core}. 
For $1\le i\le 5$, let $X_i$ be the set of vertices which have the same neighbors in $C$ as $v_i$.
Since $G$ is triangle free, these five sets are independent sets and partition $V(G)$.
In the following, all indices are calculated modulo $5$.
Vertices in $X_i$ are not adjacent to vertices in $X_{i+2}$ since $v_{i+1}$ is their common neighbor, 
for all $1 \leq i \leq 5$. 
Therefore, for any $u_i\in X_i$ and $u_{i+1}\in X_{i+1}$, $N(u_i)\cap N(u_{i+1})=\emptyset$.

If $u_i\in X_i$ and $u_{i+1}\in X_{i+1}$, then $u_iu_{i+1}\in E(G)$. Otherwise, we could add $u_iu_{i+1}$, creating a new $C_5$ in $u_iu_{i+1}v_{i+2}v_{i+3}v_{i+4}u_i$ without creating a triangle, contradicting the maximality of $C_5$.
Hence, $X_i\cup X_{i+1}$ induces a complete bipartite graph, and $G$ is a blow-up of $C_5$.

Denote $|X_i|$ by $x_i$ for $1 \leq i \leq 5$, so $n=x_1+x_2+x_3+x_4+x_5$. Then the number of $C_5$ in $G$ is $x_1x_2x_3x_4x_5$. 
This number is maximized if and only if $|x_i-x_j|\le 1$ for all $0\le i\le j\le 5$. Therefore, $G$ is a balanced blow-up of $C_5$, proving Theorem~\ref{main}.
\end{proof}

Denote by $C_5^+$ the balanced blow-up of $C_5$ on $6$ vertices, i.e., the
graph obtained from $C_5$ by duplicating one vertex.
The proof of Lemma~\ref{core} uses the following results that were obtained using flag algebras.
For any triangle-free graphon $B$ of a graph:
\begin{claim}[\cite{Grzesik,HatamiHKNR13}]\label{cl:C5}
$d(C_5) \leq 0.0384$.
\end{claim}

\begin{claim}\label{lowbound}
If $d(C_5) \geq 0.034$, then 
$d(C_5^+) \geq 4.57771 \cdot (d(C_5)-0.034) +0.095058$.
\end{claim}

\begin{claim}\label{tightup}
$d(C_5^+) \geq 6 \cdot (d(C_5)-0.0384) + 0.1152$.
\end{claim}
We defer the proofs of Claims~\ref{lowbound} and~\ref{tightup} to the next section. With these claims, we are ready to prove the core lemma.

\begin{proof}[Proof of Lemma~\ref{core}]
Let $G$ be a triangle-free graph on $n \geq 10$ vertices maximizing the number of $C_5$s, and let $B$ be its graphon.
We want to show that there exists a $5$-cycle $C=v_1v_2v_3v_4v_5v_1$ in $G$ such that $V(G)=X_1\cup\cdots \cup X_5$,
where  $v \in X_i$ if $v$ has the same neighbors in $C$ as $v_i$ for $1 \leq i \leq 5$.
In $B$, this translates to
\[
\exists x_1,x_2,x_3,x_4,x_5\in [0,1]: 
\P (B[\{ x,x_1,x_2,x_3,x_4,x_5\}]\cong C_5^+)=1
\]
for $x\in [0,1]$ picked uniformly at random.

In fact, by the integrality of the number of vertices of $G$, it is sufficient to show that
\begin{equation}\label{eq:main}
\exists x_1,x_2,x_3,x_4,x_5\in [0,1]: 
\P (B[\{ x,x_1,x_2,x_3,x_4,x_5\}]\cong C_5^+)>1-\frac1n.
\end{equation}
We may condition on $B[x_1,x_2,x_3,x_4,x_5]\cong C_5$, and then, picking $x,x_1,x_2,x_3,x_4,x_5\in [0,1]$ independently uniformly at random,
\begin{align*}
\P (B[\{x,&x_1,x_2,x_3,x_4,x_5\}]  \cong C_5^+  \; | \; B[\{x_1,x_2,x_3,x_4,x_5\}]\cong C_5)\\
&=\frac{\P (B[\{x,x_1,x_2,x_3,x_4,x_5\}]\cong C_5^+\land B[\{x_1,x_2,x_3,x_4,x_5\}]\cong C_5)}{\P (B[\{x_1,x_2,x_3,x_4,x_5\}]\cong C_5)}\\
&=\frac{\frac{2}{6}d(C_5^+)}{d(C_5)}.
\end{align*}

Let $G_n$ be a balanced blow up of $C_5$ on $n$ vertices, and let $B_n$ be its graphon. Let $d_n=d_{B_n}(C_5)\le d(C_5)$.
For $10\le n<100$, we use Claim~\ref{lowbound}, and get
\[
\frac{d(C_5^+)}{3d(C_5)}\ge \frac{4.57771 \cdot (d(C_5)-0.034) +0.095058}{3d(C_5)}\ge \frac{4.57771 \cdot (d_n-0.034) +0.095058}{3d_n}.
\]
We compute $d_n$ explicitly for each $n$, and conclude that $\frac{d(C_5^+)}{3d(C_5)}>1-\frac1n$, implying~\eqref{eq:main}.

For $n\ge 100$, we use Claim~\ref{tightup} and Claim~\ref{cl:C5}, and get
\begin{align}
\frac{d(C_5^+)}{3d(C_5)}
&\ge \frac{6 \cdot (d(C_5)-0.0384) + 0.1152}{3d(C_5)}\nonumber\\
&\ge \frac{6 \cdot (d_n-0.0384) + 0.1152}{0.1152}.\label{eq:2}
\end{align}
For $n\equiv i\mod 5$, we have
\[
d_n=\frac{5!\left(\frac{n+5-i}{5}\right)^i\left(\frac{n-i}{5}\right)^{5-i}}{n^5}>0.0384\left(1-\frac{50}{n^2}  \right),
\]
where the last inequality can easily be checked for the five cases.
Therefore, using~\eqref{eq:2},
\[
\frac{d(C_5^+)}{3d(C_5)}> \frac{-6 \cdot 0.0384\frac{50}{n^2} + 0.1152}{0.1152}=1-\frac{100}{n^2}\ge 1-\frac1n,
\]
showing~\eqref{eq:main} and thus completing the proof.
\end{proof}

\section{Flag algebra calculations}
In this section we prove Claims~\ref{lowbound} and \ref{tightup}.
Claim~\ref{cl:C5} was obtained in~\cite{Grzesik} and~\cite{HatamiHKNR13} by flag algebra computations which can easily be replicated 
 by Flagmatic~\cite{flagmatic}, including certificates. 

Both Claims~\ref{lowbound} and~\ref{tightup} are based on similar flag algebra calculations, but they exceed the current limitations of Flagmatic.
We describe an outline here. 
The actual calculations are computer assisted. 
The source code of our software to recreate the calculations is available on arXiv and
at \URL, including all data files.

The following equations are valid for all graphons. 
\begin{align}
d(C_5^+)  &= \sum_{F \in \mathcal{F}_\ell} c_F d(F), \label{eqA}\\
0          &\geq    \sum_{F \in \mathcal{F}_\ell} -\alpha_F d(F), \label{eqB}\\
0          &\geq  - \sum_{\sigma} \llbracket x^T_\sigma M_\sigma  x_\sigma \rrbracket_{\sigma},\label{eqC}
\end{align}
where $\ell$ is an integer at least $6$, $\mathcal{F}_\ell$ is the set of all triangle-free graphs on $\ell$ vertices, $c_F$ are suitably picked rational non-negative coefficients, $\alpha_F$ are any non-negative real numbers, $x_\sigma$ are vectors of some flag densities of type $\sigma$, $M_\sigma$ are any positive semidefinite matrices, and $\llbracket . \rrbracket$ is the unlabeling operator.

In addition, if the graphon satisfies $d(C_5) \geq LB$ for some fixed number $LB$, then for
any non-negative real number $y$ holds
\begin{align}
0          &\geq  - y \cdot (d(C_5) - LB).\label{eqD}
\end{align}

Now we sum all the equations \eqref{eqA} to \eqref{eqD} and we get
\begin{align}
d(C_5^+) &\geq \left[\sum_{F \in \mathcal{F}_\ell} (c_F-\alpha_F) d(F)-\sum_{\sigma} \llbracket x^T_\sigma M_\sigma  x_\sigma \rrbracket_{\sigma}
-y \cdot d(C_5) \right]  + y \cdot LB,
\label{eqF}
\end{align}
The maximum of the right-hand side can be obtained by solving a semidefinite program in variables $\alpha_F$, $M_\sigma$, and $y$.
Notice that for any assignment of the variables,  \eqref{eqF} is valid for any graphon satisfying $d(C_5) \geq LB$. 
So for a fixed assignment of variables, we can change $LB$ and obtain different lower bounds on $d(C_5^+)$,
albeit not necessarily optimal bounds. 

We optimize \eqref{eqF}  for $LB \in \{0.034,0.0384\}$ and obtain Claims~\ref{lowbound} and \ref{tightup} 
by rounding the solution to the corresponding semidefinite program.
The rounding for Claim~\ref{lowbound} is easy as the bound is not tight.
The bound in Claim~\ref{tightup} is tight for $d(C_5)=0.0384$, and the rounding of the semidefinite program required a bit more care.

\bibliographystyle{abbrv}
\bibliography{references.bib}

\end{document}